\documentclass{amsart}
\usepackage[curve, knot, dvips, arrow, all, arc]{xy}
\usepackage{amssymb,amsmath,amscd,graphicx,textcomp}

\newtheorem{theorem}{Theorem}[section]

\newtheorem{remark}[theorem]{\it Remark}
\newtheorem{example}[theorem]{Example}
\newtheorem{proposition}[theorem]{Proposition}

\xyoption{arrow}

\xyoption{matrix}

\def\C{\mathbb{C}}
\def\R{\mathbb{R}}

\def\tree{\mathcal{T}}

\title{Dissimilarity maps on trees and and the representation theory of $SL_m(\C)$}
\author{Christopher Manon}
\thanks{This work was supported by the NSF fellowship DMS-0902710}

\begin{document}

\begin{abstract}
We prove that $m$-dissimilarity vectors of weighted trees are points
 on the tropical Grassmannian, as conjecture by Cools in response to a question of 
Sturmfels and Pachter.  We accomplish this by relating $m$-dissimilarity vectors
to the representation theory of $SL_m(\C).$
\end{abstract}

\maketitle

\tableofcontents

\smallskip

\section{Introduction}

We will explore tropical properties of weighted, or metric trees
$\tree,$  using the representation theory of the special linear group $SL_m(\C).$
We direct the reader to the book by Fulton and Harris \cite{FH} and the book
by Dolgachev \cite{D} for an introduction
to the representation theory of connected complex reductive groups over $\C$. 
Recall that we can choose a Borel subgroup $B$, and a maximal torus $T$
with $T \subset B \subset G,$ and associate to this data a monoid of weights
$C_G \subset X(T)$ in the characters of $T,$ which classify irreducible representations
of $G$ up to isomorphism. This cone comes with an involution defined by the duality operation on representations $\lambda \to \lambda^*.$  The direct sum of all such representations forms a commutative algebra 

\begin{equation}
R(G) = \bigoplus_{\lambda \in C_G} V(\lambda^*)
\end{equation}

\noindent
which is the coordinate ring of the quotient of $G$ by the unipotent radical of
a chosen Borel subgroup, $R(G) = \C[G/U].$   For $SL_m(\C)$, this can be taken
to be the subgroup of unipotent upper-triangular matrices.  Choosing a Borel subgroup
also fixes a set of positive roots $R_+ \subset X(T),$ for $G,$ which define a partial 
ordering on the weights, we say that $\lambda \geq \lambda'$ if $\lambda - \lambda'$
is a member of $\mathbb{N}R_+.$  For $SL_m(\C),$ the cone $C_{SL_m(\C)}$ is generated over $\mathbb{Z}_+$  by $m-1$ fundamental weights, $\omega_1, \ldots, \omega_{m-1}.$  The weight $\omega_i$ is the so-called highest weight of the representation $\bigwedge^i(\C^m).$  The main result of this paper expresses the $m$-dissimilarity vector of an arbitrary tree in terms of the fundamental weights of $SL_m(\C)$.  In what follows $A_+$ denotes the non-negative members of $A = \mathbb{Z}$ or $\mathbb{R}.$

\subsection{Dissimilarity maps and the Grassmannian}

Let $\tree$ be a trivalent tree with $n$ ordered leaves, and let $\ell: Edge(\tree) \to \mathbb{R}_+$ be a function which assigns a weight (or length) to each edge of $\tree.$  The weight function $\ell$ defines a metric on the leaves $L(\tree)$ of $\tree,$ where the distance $d_{ij}$ between the leaves $i$ and $j$ is the sum of the weights on the edges of the unique path $\gamma(i,j)$ connecting $i$ to $j$ in $\tree.$ We intentionally confuse the path $\gamma_{ij}$ with the set of edges it traverses.

\begin{equation}
d_{ij}(\tree, \ell) = \sum_{e \in \gamma(i,j)} \ell(e)\\
\end{equation}

\noindent
Obviously $d_{ij} = d_{ji}$ and $d_{ii} = 0.$  We call the vector $D_{2, n}(\tree, \ell) = \{d_{ij}(\tree, \ell)\}_{i < j} \in \R^{\binom{n}{2}}$ the $2$-dissimilarity vector of $(\tree, \ell).$  We may generalize this construction by introducing the convex hull of $m$ leaves $i_1 \ldots i_m \in L(\tree)$ as the set of all edges which appear in paths connecting some $i_j$ to some $i_k.$

\begin{equation}
\gamma(i_1 \ldots i_m) = \bigcup  \gamma(i_ji_k)\\
\end{equation}

\noindent
The $m$-dissimilarity vector $D_{m, n}(\tree, \ell) = \{d_{i_1 \ldots i_m}(\tree, \ell)\}_{i_1 < \ldots < i_m} \in \R^{\binom{n}{m}}$ is then defined as expected.

\begin{equation}
d_{i_1 \ldots i_m}(\tree, \ell) = \sum_{e \in \gamma(i_1 \ldots i_m)} \ell(e)\\
\end{equation}

\begin{figure}[htbp]
\centering
\includegraphics[scale = 0.35]{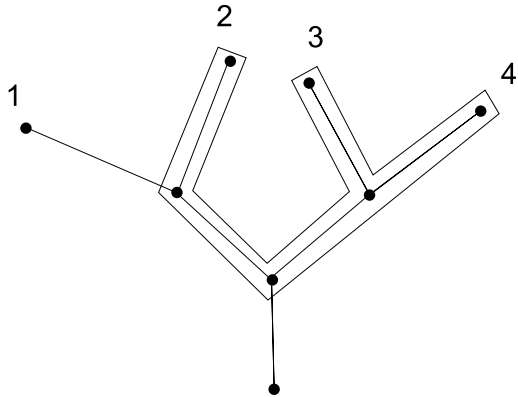}
\caption{The convex hull of three leaves}
\label{fig:1}
\end{figure}

The set of $2$-dissimilarity vectors of weighted trees $T_{2, n} \subset \R^{\binom{n}{2}}$  is well understood.  A weighted tree can be recovered from its $2$-dissimilarity vector, and the set of all $2$-dissimilarity vectors is characterized by the following theorem from tropical geometry, see \cite{SpSt}, \cite{C}.

\begin{theorem}
The set of $2$-dissimilarity vectors coincides with the tropical Grassmanian.

$$T_{2, n} = Trop(Gr_2(\C^n)) \subset \R^{\binom{n}{2}}$$

\end{theorem}
\noindent
The Gr\"obner fan of the Pl\"ucker algebra $\C[X_{i_1, \ldots, i_m}] / I_{m, n}$ has support $\R^{\binom{n}{m}},$ so each vector $w$ in this space gives a Gr\"obner degeneration of the ideal $I_{m, n}$ to the ideal of initial forms $in_w(I_{m, n})$.  The set of vectors $w$ which for which $in_w(I_{m, n})$ is monomial free is called the tropical Grassmannian $Trop(Gr_m(\C^n)).$   The above theorem implies that for a vector $w \in \R^{\binom{n}{2}}$ to be a $2$-dissimilarity vector, it must weight Pl\"ucker variables $z_{ij}$ in such a way that at least two monomials in each Pl\"ucker relation 

\begin{equation}
z_{ij}z_{k\ell} - z_{ik}z_{j\ell} + z_{i\ell}z_{jk}\\
\end{equation}
\noindent
have the same weight.  For a point $\vec{w} = \{w_{ij}\}$ to satisfy this requirement,
the maximum of 
$\{w_{ij} + w_{k\ell}, w_{ik} + w_{j\ell}, w_{i\ell} + w_{jk}\}$ must be obtained
at least twice. If this is the case, then we may find a tree $(\tree, \ell)$ such
that $d_{i, j}(\tree, \ell) = w_{ij}.$ Since $2$-dissimilarity vectors characterize
their respective weighted trees, we should expect that some operation on the $2$-dissimilarity
vector  of a weighted tree $(\tree, \ell)$, probably tropical in nature, will yield the
$m$-dissimiliarity vector, and indeed this is the case, see \cite{BC} for the following
theorem. 

\begin{theorem}
Let $\mathcal{C}_m$ be the set of length $m$ cycles in the set of permutations on $m$ letters.
Then we have the following formula. 

\begin{equation}
d_{i_1 \ldots i_m}(\tree, \ell) = \frac{1}{2} min_{\alpha \in C_m} \{d_{i_1 i_{\alpha(1)}} (\tree, \ell) + \ldots +d_{i_{\alpha^{m-1}(1)} i_{\alpha^m(1)}
}(\tree, \ell) \} \\
\end{equation} 

\noindent
This defines an onto map $\phi^{(m)}: T_{2, n} \to T_{m, n}.$
\end{theorem}

Given that $T_{m, n}$ and the tropical Grassmannian $Trop(Gr_m(\C^n))$ live
in the same space, and coincide for $m = 2,$ one would hope that these two
sets always have a close relationship. Sturmfels and Pachter asked
if the set of $m$-dissimilarity vectors was always contained in the
tropical Grassmannian $Trop(Gr_m(\C^n,)$ \cite{PSt}. Cools recently proved 
this for small $m,$ \cite{C} and conjectured that the result holds for all $m,$ the
result was proved in general by Giraldo, \cite{G}.

\begin{theorem}[Cools, Giraldo]\label{main}
\begin{equation}
\phi^{(m)}(T_{2, n}) = T_{m, n} \subset Trop(Gr_m(\C^n))\\ 
\end{equation}
\end{theorem}
\noindent
This means that the entries of $D_{m, n}(\tree, \ell)$ always satisfy
the tropical Pl\"ucker equations, and the weighting defined by this vector defines a monomial free initial ideal $in_{D_{m,n}(\tree, \ell)}(I_{m,n}).$ The purpose of this note is to prove this theorem using tropical properties of the Pl\"ucker algebra deduced from the related representation theory of $SL_m(\C)$.

\subsection{Invariants in tensor products of representations.}

The Pl\"ucker algebra $\C[X_{i_1, \ldots, i_m}]/ I_{m, n}$ is a natural
object in the representation theory of $SL_m(\C),$ it appears as the subring of invariants of the diagonal action of $SL_m(\C)$ on $M_{m \times n}(\C),$ this is the First Fundamental Theorem of Invariant Theory.

\begin{equation}
\C[X_{i_1, \ldots, i_m}]/ I_{m, n} \cong \C[M_{m\times n}(\C)]^{SL_m(\C)}\\
\end{equation}
\noindent
The Pl\"ucker algebra is exactly the subring generated by the Pl\"ucker coordinates,
$Z_{i_1 \ldots i_m}$.  Let $A \in M_{m\times n}(\C),$ $A = [C_1 \ldots C_n],$ with $C_i \in \C^m,$ then the value of the Pl\"ucker coordinates at $A$ is defined by the following. 

\begin{equation}
Z_{i_1 \ldots i_m}(A) = det[C_{i_1} \ldots C_{i_m}]\\
\end{equation}
\noindent
We may rewrite this algebra in terms of the category of finite dimensional representations
of $SL_m(\C)$ as follows.

\begin{equation}
\C[M_{m \times n}(\C)]^{SL_m(\C)} = \bigoplus_{\vec{r} \in \mathbb{Z}_+^n} [V(r_1\omega_1^*) \otimes \ldots \otimes V(r_n \omega_1^*)]^{SL_m(\C)}\\
\end{equation}
\noindent
Here $\omega_1$ is the highest weight of $\C^m$ as a representation of $SL_m(\C)$ and $\omega_1^* = \omega_{m-1}$.  With respect to this direct-sum decomposition, the Pl\"ucker coordinate $Z_{i_1 \ldots i_m}$ is a generator of the summand with $V(\omega_1^*)$ in the $r_{i_j}-$th place for all $i_j \in \{i_1, \ldots, i_m\},$ and the trivial representation everywhere else.  Multiplication in $\C[M_{m, n}(\C)]$ has a nice description in terms of this direct sum decomposition as well, it is induced by the Cartan multiplication maps in each component, where the tensor product is projected onto its highest weight summand.

\begin{equation}
V(r_1\omega_1^*) \otimes V(r_2 \omega_1^*) \to V((r_1 + r_2) \omega_1^*)  \\
\end{equation}

We may rewrite each summand in terms of homomorphisms from the category
of $SL_m(\C)$ representations.
\begin{equation}
[V(r_1\omega_1^*) \otimes \ldots \otimes V(r_n\omega_1^*)]^{SL_m(\C)} = Hom_{SL_m(\C)}(\C, V(r_1\omega_1^*) \otimes \ldots \otimes V(r_n\omega_1^*))\\
\end{equation}
\noindent
In this way, the Pl\"ucker algebra encodes the branching problem of finding copies of the trivial representation of $SL_m(\C)$ in an irreducible representation of $SL_m(\C)^n.$    
For this reason, we will refer to the Pl\"ucker algebra as a  branching algebra.  In general, a branching algebra encodes the branching rules of irreducible representations for some map of reductive groups. In this case the map is the diagonal map $\Delta_n: SL_m(\C) \to SL_m(\C)^n.$  Filtrations and associated graded algebras of branching algebras like this one were studied by the author in \cite{M}, in particular for diagonal embeddings as above, the author described a way to produce filtrations of the branching algebra associated to labelled, rooted trees.  We will review the details of this construction in the next section, but for now we will describe the features that we need.  Let $C_{SL_m(\C)}$ be the cone of dominant weights with respect to the standard ordering of weight vectors in the weight lattice
for $SL_m(\C).$  Let $\hat{\tree}$ be a rooted tree with $n$ leaves, 
we consider the orientation induced on the edges of $\hat{\tree}$ by orienting every edge in the unique path from the root to a leaf in such a way to make the root the unique source. 

\begin{proposition}\label{decomp}
Let $\hat{\tree}$ be a rooted tree with $n$ leaves.  There
is a direct-sum decomposition,

\begin{equation}
[V(r_1 \omega_1) \otimes \ldots \otimes V(r_n \omega_1)]^{SL_m(\C)} = \bigoplus \mathcal{W}(\hat{\tree}, \lambda)\\
\end{equation}
\noindent
over all $\lambda: Edge(\hat{\tree}) \to C_{SL_m(\C)},$ such that the root edge is weighted $0,$  the edge incident to the $i$'th leaf is weighted $r_i \omega_1^*,$ and for each internal vertex, the representation associated to the label on the sink appears in the direct sum decomposition of the tensor product of the representations associated to the labels on the sources.  The summand $\mathcal{W}(\hat{\tree}, \lambda)$ is the vector space of all possible assignments of intertwiners to the internal vertices which realize the weight on a source at a vertex as a summand of the tensor product of the weights on sinks.
\end{proposition}

\noindent
 In figure \ref{fig:2} we an example of such an object with $SL_m(\C)$ representations given by Young tableaux.  Recall that $\omega_1^* = \omega_{m-1}.$

\begin{figure}[htbp]
\centering
\includegraphics[scale = 0.3]{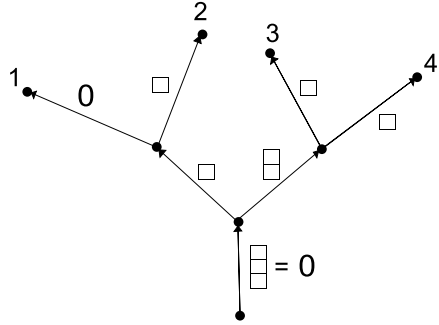}
\caption{A tree weighted with representations}
\label{fig:2}
\end{figure}

Proposition \ref{decomp} is a formal consequence of properties of semisimple categories with
monoidal products.  The tree $\hat{\tree}$ can be considered as a recipe for inserting
parentheses into the tensor product $V(r_1\omega_1^*) \otimes \ldots \otimes V(r_n\omega_1^*),$
and gives a way to recursively expand the expression into a direct sum. We can then take the same tree $\hat{\tree}$ and assign to its edges $e \in Edge(\tree)$ functionals $h_e: C_{SL_m(\C)} \subset X(T_{SL_m(\C)}) \to \mathbb{R}_+,$  
taking care that $h_e$ is positive on all positive roots.  We apply this functional $(\hat{\tree}, \hat{h})$ to each summand,  where an element in $\mathcal{W}(\hat{\tree}, \lambda)$
is given the weight $\sum_{e \in Edge(\hat{\tree})} h_e(\lambda(e)).$  This weights each Pl\"ucker coordinate $Z_{i_1 \ldots i_m}$ with a number dependent on $\hat{h}$ and the tree $\hat{\tree},$ and so gives a point in $\R^{\binom{n}{m}}.$  After reviewing the construction of this filtration and understanding it with respect to the multiplication operation in the Pl\"ucker algebra, we will be able to conclude the following.

\begin{theorem}
Each $(\hat{\tree}, \hat{h})$ defines a point in $Trop(Gr_m(\C^n)) \subset \R^{\binom{n}{m}}$
\end{theorem}

This will follow from general arguments on filtrations of branching algebras obtained
from the associated representation theory, in particular we will give a general way to produce points on the tropical varieties of ideals defining these algebras.  The functionals $(\hat{\tree}, \hat{h})$ have a good amount of flexibility, enough to show the following theorem.

\begin{theorem}
There exists for any weighted tree $(\tree, \ell)$ a tree functional $(\hat{\tree}, \hat{h})$
such that $d_{i_1\ldots i_m}(\tree, \ell) = (\hat{\tree}, \hat{h})(Z_{i_1 \ldots i_m})$ for 
all $m$ tuples $\{i_1, \ldots, i_m\}.$  In particular, $m$-dissimilarity vectors are points on the tropical Grassmannian. 
\end{theorem}
\noindent 

\subsection{Acknowledgements}
We thank the reviewer for several useful suggestions, including
the example \ref{rev}. 

\section{Filtrations of branching algebras}

In this section we will review the construction of filtrations of branching algebras
introduced in \cite{M}.  The basic object we will be working with is the algebra
$R(G) = \C[G]^{U_G} = \bigoplus_{\lambda \in C_G} V(\lambda^*),$ where $G$ is a
connected reductive group over $\C$, $U_G \subset G$ is a maximal unipotent subgroup, 
$\lambda$ are dominant weights, $V(\lambda)$ is the irreducible representation with 
highest weight $\lambda,$ and $C_G$ is the monoid of dominant weights.
We choose highest weight vectors for each irreducible representation $v_{\lambda} \in V(\lambda).$  Multiplication in $R(G)$ is induced by Cartan multiplication, see 
\cite{AB} for an introduction to the algebra $R(G)$.

$$
\begin{CD}
V(\alpha^*)\otimes V(\beta^*) @>C_*>> V(\alpha^* + \beta^*)\\
\end{CD}
$$

\noindent
Identify $V(\lambda^*)$ with the dual $V(\lambda)^*$ in the unique
way that makes $ev(v_{\lambda}, \hat{v}_{\lambda^*}) = 1$ where 
$ev:V(\lambda)\otimes V(\lambda)^* \to \C$ sends $v\otimes f$
to $f(v),$ and $\hat{v}_{\lambda^*}$ is the lowest weight vector
of $V(\lambda^*).$  Under this identification,
Cartan multiplication is the dual of the  map which sends
$v_{\alpha + \beta}$ to $v_{\alpha} \otimes v_{\beta}$.

$$
\begin{CD}
V(\alpha + \beta) @>C^*>> V(\alpha) \otimes V(\beta)\\
\end{CD}
$$
\noindent
Let $\phi: H \to G$ be a map of connected reductive groups over $\C,$
we define the branching algebra $\mathfrak{A}(\phi)$ of $\phi$ as follows.

\begin{equation}
\mathfrak{A}(\phi) = [R(H) \otimes R(G)]^H = \bigoplus_{(\alpha, \beta) \in C_H\times C_G} [V(\alpha^*)\otimes W(\beta^*)]^H\\
\end{equation}
\noindent
Here $H$ acts on $R(G)$ through $\phi,$ and $\phi$  maps $U_H$ to $U_G.$
Branching algebras are so-named because the dimension of their multigraded components give
the branching multiplicities for irreducible representations of $G$ as representations of $H.$
We will now rewrite the multiplication operation in $\mathfrak{A}(\phi)$ with respect to the following identity. 

\begin{equation}
[V(\alpha^*)\otimes W(\beta^*)]^H = Hom_H(\C, V(\alpha^*)\otimes W(\beta^*)) \cong Hom_H(V(\alpha), W(\beta^*))\\
\end{equation}
\noindent
The isomorphism on the right is given by the following construction, for $f \in Hom_H(\C, V(\alpha^*)\otimes W(\beta^*)) $.

$$
\begin{CD}
V(\alpha) = V(\alpha)\otimes \C @>id \otimes f>> V(\alpha)\otimes V(\alpha^*)\otimes W(\beta^*) @>ev \otimes id>> W(\beta^*)\\
\end{CD}
$$
\noindent
Let $\hat{f} = (ev \otimes id) \circ (id \otimes f)$ denote the transformed map. Under this isomorphism, the multiplication map
$$
\begin{CD}
\C \otimes \C @>f \otimes g>> [V(\alpha_1^*)\otimes W(\beta_1^*)] \otimes [V(\alpha_2^*)\otimes W(\beta_2^*)] @>C_* \otimes C_*>> V(\alpha_1^* + \alpha_2^*) \otimes W(\beta_1^* + \beta_2^*)\\
\end{CD}
$$
\noindent
becomes
$$
\begin{CD}
V(\alpha_1 + \alpha_2) @>C^*>> V(\alpha_1)\otimes V(\alpha_2) @>\hat{f} \otimes \hat{g}>> W(\beta_1^*)\otimes W(\beta_2^*) @>C_*>> W(\beta_1^* + \beta_2^*)\\
\end{CD}
$$
\noindent
this is a straightforward calculation. Now we consider a factorization of $\phi$ in the category of connected reductive groups over $\C.$
$$
\begin{CD}
H @>\psi>> K @>\pi>> G\\
\end{CD}
$$
\noindent
We formally get a direct sum decomposition of each multigraded component of the branching algebra $\mathfrak{A}(\phi).$

\begin{equation}
Hom_H(V(\alpha), W(\beta^*)) = \bigoplus_{\eta \in C_K} Hom_H(V(\alpha), Y(\eta^*)) \otimes Hom_K(Y(\eta^*), W(\beta^*))\\
\end{equation}
\noindent
This introduces a host of combinatorial representation theory data into the algebra
$\mathfrak{A}(\phi).$  We will see how to multiply two elements, we start 
by taking the tensor product.
$$
\begin{CD}
V(\alpha_1) \otimes V(\alpha_2) @>f_1 \otimes f_2>> Y(\eta_1^*) \otimes  Y(\eta_2^*) @>g_1\otimes g_2>> W(\beta_1^*) \otimes W(\beta_2^*)\\
\end{CD}
$$
\noindent
The middle representation decomposes as a direct sum of $K$ representations, 

\begin{equation}
Y(\eta_1^*)\otimes Y(\eta_2^*) = \bigoplus_{\eta \in C_K} Hom_K(Y(\eta^*), Y(\eta_1^*)\otimes Y(\eta_2^*)) \otimes Y(\eta^*)\\
\end{equation}
\noindent
this allows us to represent $f_1 \otimes f_2$ and $g_1 \otimes g_2$ as
sums of maps.  Let $\pi_{\eta}: Y(\eta_1^*)\otimes Y(\eta_2^*) \to Hom_K(Y(\eta^*), Y(\eta_1^*)\otimes Y(\eta_2^*)) \otimes Y(\eta^*)$ and $\pi^{\eta}: Hom_K(Y(\eta^*), Y(\eta_1^*)\otimes Y(\eta_2^*)) \otimes Y(\eta^*) \to Y(\eta_1^*)\otimes Y(\eta_2^*)$
be projections and injections that define the direct sum decomposition with $\pi_{\eta_1 + \eta_2} = C_*$ and $\pi^{\eta_1 + \eta_2} = C^*.$  Then we have

\begin{equation}
f_1 \otimes f_2 = \sum \pi_{\eta} \circ (f_1 \otimes f_2), \\
\end{equation}
\begin{equation}
g_1 \otimes g_2 = \sum (g_1 \otimes g_2) \circ \pi^{\eta}\\
\end{equation}

\noindent
Decomposing the diagram along these sums gives an expansion of the product
into components from the direct sum decomposition of $Hom_H(V(\alpha_1 + \alpha_2), W(\beta_1 + \beta_2))$, and there is a natural leading term given by the sum of the weights,
$$
\begin{CD}
V(\alpha_1 + \alpha_2) @>C_*\circ(f_1 \otimes f_2) \circ C^*>> Y(\eta_1^* + \eta_2^*) @> C_* \circ (g_1 \otimes g_2) \circ C^*>> W(\beta_1^* + \beta_2^*).\\
\end{CD}
$$
\noindent
A general term,
$$
\begin{CD}
V(\alpha_1 + \alpha_2) @> \pi_{\eta} \circ(f_1 \otimes f_2) \circ C^*>> Y(\eta^*) @> C_* \circ (g_1 \otimes g_2) \circ \pi^{\eta}>> W(\beta_1^* + \beta_2^*).\\
\end{CD}
$$
\noindent
is a member of the $(\alpha_1 + \alpha_2, \eta, \beta_1 + \beta_2)$ summand.
The leading term never vanishes, because the defining maps
$C_*\circ (f_1 \otimes f_2) \circ C^*$ and $C_* \circ (g_1 \otimes g_2) \circ C^*$
are the same as the multiplication operation in $\mathfrak{A}(\psi)$ and $\mathfrak{A}(\pi)$ respectively. These algebras are domains because $R(G)$ is always a domain.  Notice that this analysis depends only on multigraded summands of $\mathfrak{A}(\phi),$ so 
the same term decomposition exists for any subalgebra which preserves the multigrading.  We summarize the previous discussion.

\begin{proposition}\label{highest}
For any factorization of a map of connected, reductive groups over $\C,$
$$
\begin{CD}
H @>\psi>> K @>\pi>> G
\end{CD}
$$
there is a direct sum decomposition of $\mathfrak{A}(\phi)$ into
summands $\mathcal{W}(\alpha, \eta, \beta),$ with $\alpha \in C_H,$
$\eta \in C_K$ and $\beta \in C_G$ dominant weights.  This defines a 
 multifiltration of the branching algebra $\mathfrak{A}(\pi \circ \psi).$
The product of two elements
$$
\begin{CD}
V(\alpha_1) @>f_1>> Y(\eta_1) @>g_1>> W(\beta_1)\\
V(\alpha_2) @>f_2>> Y(\eta_2) @>g_2>> W(\beta_2)\\
\end{CD}
$$
has leading term
$$
\begin{CD}
V(\alpha_1+\alpha_2) @>C_* \circ f_1\otimes f_2\circ C^*>> Y(\eta_1 + \eta_2) @>C_* \circ g_1 \otimes g_2 \circ C^*>> W(\beta_1 + \beta_2).\\
\end{CD}
$$
all lower terms involve $\eta \in C_K$ which are less than $\eta_1 + \eta_2$ as dominant 
weights.
\end{proposition}

We can perform this same construction on a factorization of any length 
$$
\begin{CD}
H @>\psi>> K_1 @>\pi_1>> \ldots @>\pi_{k-1}>> K_k @>\pi>> G\\
\end{CD}
$$
\noindent
without altering the details, and proposition \ref{highest}
holds for the resulting multifiltration.  We may use
this extra combinatorial data to describe filtrations of $\mathfrak{A}(\phi).$
To each new summand $\mathcal{W}(\alpha, \vec{\lambda}, \beta)$  in the filtration, we attach a number as follows, pick functionals 

\begin{equation}
h_0: X(T_H) \to \mathbb{R}_+,\\
h_1: X(T_{K_1}) \to \mathbb{R}_+,\\
\dots,\\
h_k: X(T_{K_k}) \to \mathbb{R}_+,\\
h_{k+1}: X(T_G) \to \mathbb{R}_+,\\
\end{equation}
\noindent
such that $h_i$ has non-negative value on all positive roots of $K_i$.
Now apply these functionals to the weights defining the multifiltered summands, this defines
a filtration. 
$$
\begin{CD}
\vec{h}(V(\lambda_0) @>f_1>> V(\lambda_1) @>f_2>> \ldots @>f_k>> V(\lambda_k)) = \sum h_i(\lambda_i)\\
\end{CD}
$$
\noindent
 By proposition \ref{highest}, the value on a product of elements, computed by summing
up the contributions from each element, is always equal to the value on its leading term for any linear functional $\vec{h}.$ 

\begin{proposition}\label{pres}
Let $\Phi:\C[X] \to \mathfrak{A}(\phi)$ be a presentation of the branching algebra,
and let $\phi = \psi_1 \circ \ldots \psi_k$ be a factorization of $\phi.$
Suppose each $x \in X$ is mapped to an element of one of the summands $\mathcal{W}(\vec{\lambda}) \subset \mathfrak{A}(\phi)$ defined by the factorization, and let $I \subset \C[X]$ be the defining ideal.  Then any functional $\vec{h}$ defines a term weighting of $X$ which gives a monomial free initial ideal $in_{\vec{h}}(I)$.
\end{proposition}

\begin{proof}
Pick any expression in the ideal $I$.

\begin{equation}
F(X) = \sum c_a\vec{x}^{\vec{a}}\\
\end{equation}
\noindent
 We consider the expansion of each monomial term into
pure terms, $\Phi\circ (\vec{x}^{\vec{a}}) = S_0^a + \ldots + S_m^a,$ where $S_0^a$ has
the same pure filtration level as the monomial, the existence of this term
follows from proposition \ref{highest}, which also implies that we must have $\vec{h}(S_0) \geq \vec{h}(S_i)$ for every term in this expansion.  In general, for pure terms $X$ and $Y$, we say that $X \geq Y$ if for each component $\lambda_i(X) - \lambda_i(Y)$ is a positive  root. Note that not all pure terms are comparable. By definition of the functional $\vec{h}$ if $X \leq Y$ then $\vec{h}(X) \leq \vec{h}(Y).$  
Now suppose some monomial $\vec{x}^{\vec{a}}$ in the expression $F(X)$ 
has the highest filtration weight with respect to $\vec{h}.$ 
We must have $\Phi\circ F(X) = 0,$ so $S_0^{\vec{a}}$ must
be canceled by pure terms from the expansion of other monomials.
This implies that some monomial $\vec{x}^{\vec{b}}$ must have a pure term
$S_j^{\vec{b}}$ with the same multifiltration level as $S_0^{\vec{a}}.$
We must have that $S_0^{\vec{a}} \leq S_0^{\vec{b}}$ as pure terms, by assumption this implies that $\vec{x}^{\vec{b}}$ has the same filtration weight as $\vec{x}^{\vec{a}}$. 
\end{proof}

This proposition implies that every $\vec{h}$ defines a point on the tropical
variety of the defining ideal $I.$  It also implies that for any presentation
$\Phi: \C[X] \to \mathfrak{A}(\phi),$ and any form in the defining ideal $F(X) \in I,$
the leading terms of at least two monomials agree, a result independent of a
functional $\vec{h}.$   The functionals $\vec{h}$ fit into the broader theory of 
valuations on rings.  Roughly these are functions $v$ on a ring which satisfy
$v(ab) = v(a) + v(b),$  $v(a + b) \leq max\{v(a), v(b)\},$ and $v(0)$ is $0$
or $-\infty$ depending on the tropical algebra where $v$ takes its values. 
Generally speaking valuations define "universal" tropical points, in that they 
define a point on the tropical variety of any presentation of a subring of the
ring on which they are defined.  We explore these objects in the note
\cite{M2}, see also \cite{P}.  For each factorization of $\phi: H \to G$

\begin{equation}
F = \{\psi_1, \ldots, \psi_k\},  \\
\end{equation}
\begin{equation}
\phi = \psi_1 \circ \ldots \circ \psi_k\\
\end{equation}

\noindent
we obtain a cone of functionals $\vec{h} \in P_F$ defined by the conditions
on the components of $\vec{h}.$ Note that the $h_i = 0$ is always an option, indeed this essentially forgets the information in $i$-th component of the multifiltration.  For each
factorization $F = \{\psi_1, \ldots, \psi_k\}$ and every $i,$ there is an operation

\begin{equation}
O_i(F) = \{\psi_1 \ldots, \psi_{i-1}, \psi_{i+1}\circ \psi_i, \psi_{i+2}, \ldots, \psi_k\}\
\end{equation}

\noindent
Setting $h_i$ to $0$ gives a map of cones $P_{O_i(F)} \to P_F$ which defines $P_{O_i(F)}$ as a face of $P_F.$  This defines a connected complex of cones $\bigcup_{F \circ \phi} P_F$ over all factorizations of $\phi$ in the category of connected, reductive groups.  The content of the proposition above is that there is a map from this complex into the tropical variety of any presentation of $\mathfrak{A}(\phi),$ the same holds for any subalgebra of $B \to \mathfrak{A}(\phi)$ which preserves the multigrading.  In particular, this is true for the subalgebra of invariants, which will be important in the sequel.  

\begin{equation}
R(G)^H = [\C\otimes R(G)]^H \subset [R(H)\otimes R(G)]^H = \mathfrak{A}(\phi),\\ 
\end{equation}

\begin{example}\label{rev}
We can also look at branching deformations for the trivial subgroup of a reductive group
$1 \to G.$  This morphism is factored by any flag of subgroups of $G,$ for instance
we can take $G = GL_n$ and look at the flag

\begin{equation}
1 \to GL_1 \to GL_2 \to \ldots \to GL_{n-1} \to GL_n.\\
\end{equation}

\noindent
The branching of $GL_m$ over $GL_{m-1}$ is multiplicity free, so the branching algebra
associated to this pair is toric.  Choosing a functional $h_i: \Delta_{GL_i} \to \R_+$
which is positive on positive roots then defines a toric deformation of $R(GL_n)$ to the monoid
of Gel'fand Tsetlin patterns.  
\end{example}

\begin{example}
Any representation $V$ of a reductive group $G$ defines a morphism $G \to GL_n$
for $n = dim(V).$  First we note that if $V$ is reducible, then the map factors through
$GL_{n_1} \times \ldots \times GL_{n_k} \to GL_n$  for some partition of $n.$  Also, 
the map defined by $V$ always defines a factorization of the trivial morphism 
$1 \to G \to GL_n,$ and can therefore be identified with a cone of filtrations on $R(Gl_n).$
\end{example}

\begin{remark}
One can use this technique to define degenerations of a wide range of varieties with symmetry. Let $A$ be a commutative ring with the action of a product of reductive groups $H \times G,$ and $\phi:H \to G$ be a map of connected reductive groups over $\C$.   There is a flat degeneration of $A$ defined in \cite{Gr} which preserves the action of $H \times G,$ to the algebra $[A^{U_H^- \times U_G^- }\otimes R(H)\otimes R(G)]^{T_H \times T_G}$, where $T_G$ is a maximal torus of $G.$ Taking $H$ invariants for the action of $(id, \phi):H \to H\times G$ gives a flat degeneration $A^H \rightarrow [A^{U_G} \otimes \mathfrak{A}(\phi)]^{T_G},$  this can then be composed with degenerations of the branching algebra.  This technique was adapted by the author in \cite{M} to study properties of a quantum analogue of a branching algebra coming from conformal field theory.  A similar sort of universality holds
for other types of degenerations defined from the combinatorics of representation theory, for instance toric degenerations of spherical varieties, see \cite{AB} for details.  
\end{remark}

\begin{remark}
In \cite{M} the author also studied the associated graded algebra of a
branching filtration.  For a factorization,
$$
\begin{CD}
H @>\psi>> K @>\phi>> G\\
\end{CD}
$$
\noindent
If the functional $\hat{h}$  is strictly positive on the 
positive roots of $K,$ then we get a flat deformation
over $\C[t],$

\begin{equation}
\mathfrak{A}(\phi \circ \psi) \Rightarrow [\mathfrak{A}(\psi)\otimes \mathfrak{A}(\phi)]^{T_K}\\
\end{equation}
\end{remark}

\section{Diagonal branching algebras for $SL_m(\C)$}

In this section we use the results from the previous section
to study $\C[M_{m\times n}(\C)]^{SL_m(\C)} \subset \mathfrak{A}(\Delta_n),$
where $\Delta_n: SL_m(\C) \to SL_m(\C)^n$ is the diagonal embedding. These
embeddings have a special class of filtrations classified by rooted trees with $n$ leaves. Take such a tree $\hat{\tree}$ and define a factorization of $\Delta_n$ as follows,
let $\hat{\tree}$ have the orientation induced by the root, as before. 
For each internal vertex $v \in \hat{\tree}$ attach the diagonal morphism
$\Delta_{val(v)-1}$ from one copy of $SL_m(\C)$ to $SL_m(\C)^{val{v} -1}.$

\begin{figure}[htbp]
\centering
\includegraphics[scale = 0.3]{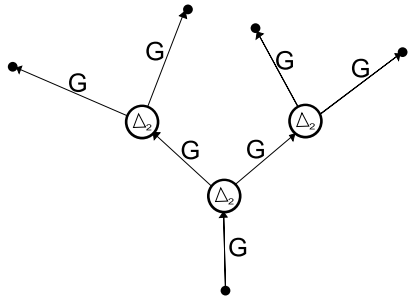}
\caption{Factorization of $\Delta_4$}
\label{fig:4}
\end{figure}

By well-ordering the non-leaf vertices of $\hat{\tree}$ in any way such that 
the first vertex is attached to the root, and two consecutive vertices share an edge
allows us to write this factorization in the style of the previous section.
$$
\begin{CD}
SL_m(\C) @>\Delta_{val(v_1)-1}>> SL_m(\C)^{val{v_1}-1} @>Id^{val(v_1)-2}\times \Delta_{val(v_2)-1}>> SL_m(\C)^{val(v_1) + val(v_2) - 3} \ldots SL_m(\C)^n\\
\end{CD}
$$
\noindent
This results in a direct sum decomposition of $\mathfrak{A}(\Delta_n)$ into 
spaces $\mathcal{W}(\hat{\tree}, \lambda)$  indexed by assignments of dominant weights of $SL_m(\C)$ to the edges of $\hat{\tree}$, along with an assignment of $SL_m(\C)-$linear maps at every vertex intertwining the corresponding tensor products of irreducible representations. From the introduction we know that the Pl\"ucker algebra is the subalgebra of $\mathfrak{A}(\Delta_n)$ generated by the unique invariants $[\C \otimes \ldots V(\omega_1^*) \otimes..\C]^{SL_m(\C)}$ where $m$ of the $n+1$ pieces of the tensor product are copies of $V(\omega_1^*) = \bigwedge^{m-1}(\C^m).$  The first piece, corresponding to the root, is always $\C,$ and the other $n-m$ pieces are the trivial representation $\C.$  Each of these spaces
is one dimensional, so we should be able to write down the
tree diagram of a basis member for a chosen $\hat{\tree}.$ To describe the diagram in general
it is simplest to start with a rooted tree $\hat{\tree}_o$ with $m$ leaves, give this tree an orientation as above.  Each leaf of this tree is labeled with $\omega_1,$ and to compute the representation labeling a given edge $e \in \hat{\tree}_o,$ simply count the number
of leaves $n_e$ above $e$ with respect to the rooted orientation, and give it the label $\omega_{n_e} = \bigwedge^{n_e}(\C^m).$ Now dualize the whole picture, so $\omega_i$
becomes $\omega_{m-i}$.  The result is shown below in Young tableaux.

\begin{figure}[htbp]
\centering
\includegraphics[scale = 0.4]{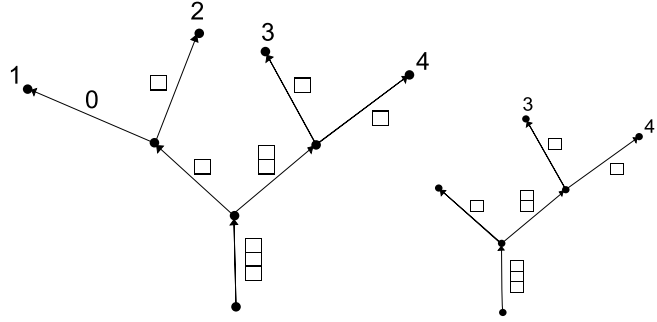}
\caption{A tree weighted with representations}
\label{fig:5}
\end{figure}

The root is labeled with $\bigwedge^m(\C^m)$ which is trivial as an $SL_m(\C)$ representation. In a general rooted tree $\hat{\tree},$ take the convex hull of the root and the non-trivially labeled leaves.  Combinatorially, this is the same as some rooted tree with $m$-leaves $\hat{\tree}_o,$ label the edges of $\hat{\tree}$ accordingly, and label all other edges with the trivial representation.  Note that up to scalars the available intertwiners $\bigwedge^{m- (i+j)}(\C^m) \to \bigwedge^{m-i}(\C^m) \otimes \bigwedge^{m-j}(\C^m)$ in this diagram are all unique as expected.  The subalgebra $\C[M_{m \times n}(\C)]^{SL_m(\C)} \subset \mathfrak{A}(\Delta_n)$ is generated by the $\binom{n}{m}$ elements of this type.  All diagrams are given explicitly in terms of the $m-1$ fundamental weights of $SL_m(\C),$ and one easily checks that an edge $e \in \hat{\tree}$ is labeled nontrivially if and only if it is in the combinatorial convex hull of the $m$-nontrivially labeled leaves. Now consider the functional $H: X(T_{SL_m(\C)}) \to \mathbb{R}_+$ defined by $H(\omega_k) = 1$ for all fundamental weights $\omega_k,$ and note that this functional gives the trivial representation the $0$ weight. Pick a non-negative length $d_e$ for edge $e \in \hat{\tree},$ and consider the functional defined by assigning $d_e H$ to the edge $e.$  For the Pl\"ucker coordinate $Z_{i_1 \ldots i_m}$ we have

\begin{equation}
(\hat{\tree}, \vec{d}H)\circ(Z_{i_1 \ldots i_m}) = \sum_{e \in conv_{\hat{\tree}}\{i_1, \ldots, i_m\}} d_e\\ 
\end{equation}

\begin{proposition}
For any metric tree $(\tree, \ell)$ with $n$ leaves there is a rooted tree
with $n$ leaves $\hat{\tree}$, and a functional $\hat{h}$ with

\begin{equation}
(\hat{\tree},\hat{h})\circ(Z_{i_1 \ldots i_m}) = d_{i_1 \ldots i_m}(\tree, \ell)
\end{equation}

\end{proposition}

\begin{proof}
To get $\hat{\tree},$ one may add a root to $\tree$ anywhere.  It is simple to verify that this preserves combinatorial convex hulls.  See figure \ref{fig:6} for an example. The root is added in the middle of an edge of $\tree,$ so in order to preserve the weighting information we must split the weight on this edge among the two new edges created by the addition of the root. 
 The previous discussion does the rest.
\end{proof}

\begin{figure}[htbp]
\centering
\includegraphics[scale = 0.35]{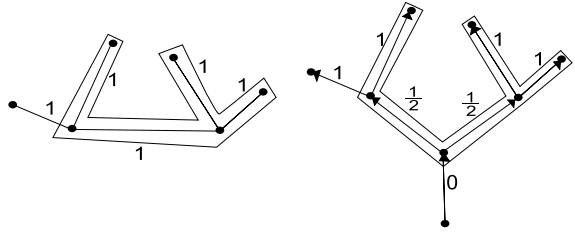}
\caption{adding a root to the tree}
\label{fig:6}
\end{figure}
\noindent
This proposition establishes that we can replicate the $m$-dissimilarity vectors
of a metric tree $(\tree, \ell)$ with branching filtrations.  The efforts of the previous
section confirm that branching filtrations always give tropical points.  Together, these
facts prove theorem \ref{main}.  

\section{Examples}

In this section we will look at dissimilarity maps, Pl\"ucker coordinates, 
, and tree weighting functionals in more detail for a specific example.
We will take a look at some elements and relations in the Pl\"ucker algebra
$\C[Gr_3(\C^8)].$  We choose a rooted tree with $8$ leaves, $\tree$. 

\begin{figure}[htbp]
\centering
\includegraphics[scale = 0.35]{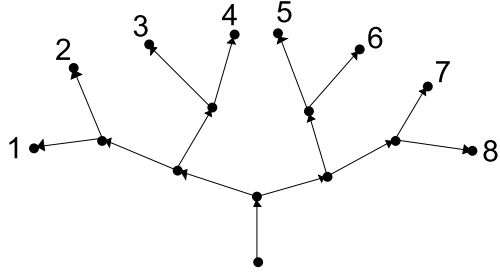}
\caption{a rooted tree with 8 leaves}
\label{fig:7}
\end{figure}

\noindent
For simplicity we give $\tree$ the metric where each edge has length $1,$
note that the corresponding unrooted tree would have one edge with length $2,$
and all others with length $1.$ We will find how $\tree$ weights the Pl\"ucker relation

\begin{equation}
Z_{123}Z_{456} - Z_{124}Z_{356} + Z_{125}Z_{346} - Z_{126}Z_{345} = 0
\end{equation}
\noindent
in $\C[M_{3\times8}(\C)]^{SL_3(\C)}.$  Each Pl\"ucker coordinate corresponds
to an assignment of representations to the edges of $\tree$,
which are then weighted with the functional $H,$ as in figure \ref{fig:9}.

\begin{figure}[htbp]
\centering
\includegraphics[scale = 0.5]{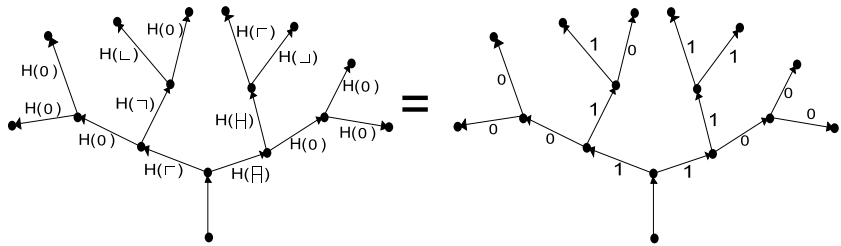}
\caption{applying functional}
\label{fig:9}
\end{figure}

\noindent
In figure \ref{fig:8} we show the convex hulls of each set of leaves, rows correspond to Pl\"ucker monomials.

\begin{figure}[htbp]
\centering
\includegraphics[scale = 0.5]{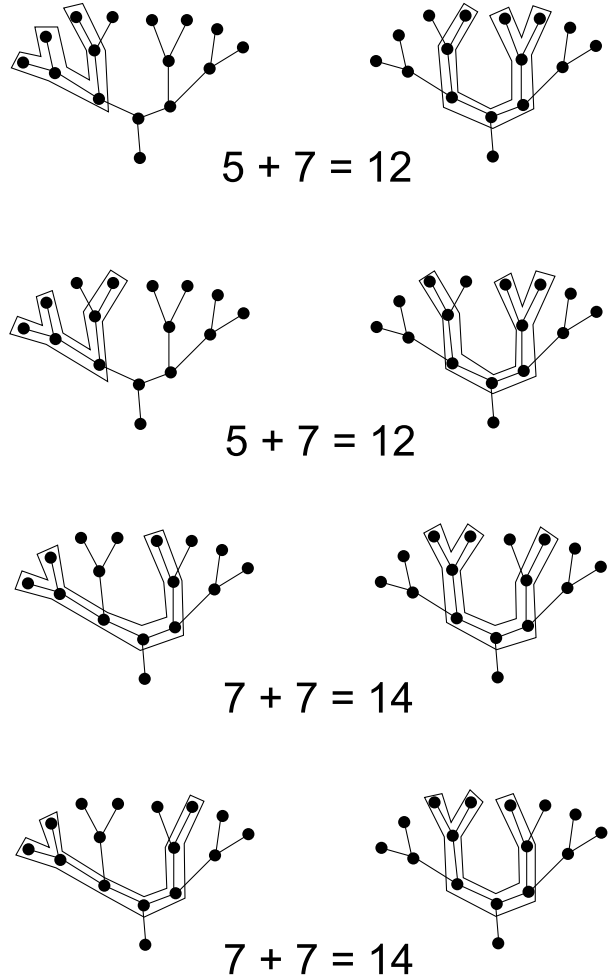}
\caption{combinatorial convex hulls of leaves}
\label{fig:8}
\end{figure}

\noindent
This results in the following weights in the Pl\"ucker relation,
\begin{equation}
t^{12}Z_{123}Z_{456} - t^{12}Z_{124}Z_{356} + t^{14}Z_{125}Z_{346} -t^{14}. Z_{126}Z_{345} = 0
\end{equation}

Next we look at the general case of $SL_2(\C).$  The $2$-dissimilarity vectors of a tree are 
the best understood dissimilarity vectors because of their association with the Grassmannian $Gr_2(\C^n),$ the same is true for $SL_2(\C^n)$ branching algebras.
The algebra $\mathfrak{A}(\Delta_n)$ for $SL_2(\C)$ is isomorphic
to $\C[M_{2\times n+1}]^{SL_2(\C)},$ indeed we have

\begin{equation}
R(SL_2(\C)) = Sym(V(\omega_1)) \cong \C[x_1, x_2]\\ 
\end{equation}
\noindent
It follows that the subalgebra of invariants $R(SL_2(\C)^n)^{SL_2(\C)}$
is isomorphic to the $(2, n)$ Pl\"ucker algebra. For a rooted tree 
$\hat{\tree},$ The functionals $(\hat{\tree}, \hat{h})$ are all given by assigning non-negative
integers to the edges of $\hat{\tree},$ as non-negative integers correspond to maps
$h_e: C_{SL_2(\C)} = \mathbb{Z}_+ \to \mathbb{R}_+.$  Therefore for any metric tree $(\tree, \ell)$ we can construct a branching algebra filtration that weights the Pl\"ucker monomials the same as $(\tree, \ell).$ In this way, every member of the tropical Grassmannian $Trop(Gr_2(\C^n))$ is realizable by a branching filtration.

\bigskip

\noindent
Christopher Manon:\\
Department of Mathematics,\\ 
University of California, Berkeley,\\ 
Berkeley, CA 94720-3840 USA,\\
chris.manon@math.berkeley.edu

\end{document}